\documentclass[11pt,a4paper]{article}

\usepackage{amssymb, latexsym, amsthm}
\usepackage{graphicx} % avoid epsfig or earlier such packages
\usepackage{url}      % for URLs and DOIs
\newcommand{\doi}[1]{\url{http://dx.doi.org/#1}}
\usepackage{amsmath}  % many want amsmath extensions

\IfFileExists{ajr.sty}{\usepackage{ajr}}{}

\newcommand{\D}{\Delta}

\renewcommand{\phi}{\varphi}

\newcommand{\R}{{\mathbb R}}

\newcommand{\pde}{\textsc{pde}}
\newcommand{\sde}{\textsc{sde}}

\newtheorem{theorem}{Theorem}
\newtheorem{lemma}[theorem]{Lemma}

\newtheorem{assumption}[theorem]{Assumption}
\theoremstyle{remark}

\topsep=\parskip

\title{On the approximation for singularly perturbed stochastic wave equations}

\author{
Wei Wang\thanks{School of Mathematics, University of Adelaide, South Australia~5005, \textsc{Australia}. \protect\url{mailto:
w.wang@adelaide.edu.au}  and Department of Mathematics, Nanjing University, Nanjing, \textsc{China}.   \protect\url{mailto:
wangweinju@yahoo.com.cn}  }
\and
Yan Lv\thanks{School of Science, Nanjing University of Science \&
Technology, Nanjing, 210094, \textsc{China}.
\protect\url{mailto:lvyan1998@yahoo.com.cn} }
\and
A.~J.
Roberts\thanks{School of Mathematics, University of Adelaide, South
Australia~5005, \textsc{Australia}. \protect\url{mailto:
anthony.roberts@adelaide.edu.au}}
}

\date{\today}

\begin{document}

% Use default \verb|\maketitle|.
\maketitle

% Use the \verb|abstract| environment.
\begin{abstract}
We explore the relation between fast waves, damping and imposed noise for different scalings by considering the singularly perturbed stochastic nonlinear wave equations
\begin{equation*}
 \nu u_{tt}+u_t=\D u+f(u)+\nu^\alpha\dot{W}
 \end{equation*}
on a bounded spatial domain. An asymptotic approximation to the stochastic wave equation is constructed by a special transformation and splitting of~$\nu u_{t}$.  This splitting gives a clear description of the structure of~$u$. The approximating model, for small $\nu>0$\,,  is a stochastic nonlinear heat equation for exponent $0\leq\alpha<1$\,, and is a deterministic nonlinear wave equation for exponent $\alpha>1$\,.
\end{abstract}


%\paragraph{Mathematics Subject Classifications (2000)} 60F10, 60H15, 35Q55.

%\tableofcontents

\section{Introduction}\label{sec:intro}

Our stochastic model is motivated by some material continuum in some domain $D\subset\R^{n}$, $1\leq n\leq 3$\,.  The continuum is made of `particles' with `displacement' field~$u(t,x)$ and `velocity' field~$v(t,x)$.
The motion of the particles in the continuum in a stochastic force field~$\sigma\dot W$,  motivated by Newton's law,  is assumed to be described  by the  following stochastic partial differential  equations~\cite{CF05}
\begin{eqnarray}\label{e:SWE1}
u_{t}^{\nu}(t,x)&=&v^{\nu}(t,x),\\
 \nu v_{t}^{\nu}(t,x)&=&-kv^{\nu}(t,x)+\D u^{\nu}(t,x)+f(u^{\nu}(t,x))+\sigma\dot{W}(t,x),\label{e:SWE2}\\
 u^{\nu}(0,x)&=&u_0\,,\quad
  v^{\nu}(0,x)=u_1 \,,\label{e:SWE3}
\end{eqnarray}
for times $t\geq 0$\,, and locations $x\in D$\,,  with zero Dirichlet boundary condition on~$\partial D$. Here small~$\nu$ is the `density' of the particles: we explore the singular limit as $\nu\to0$ and so label the dependent fields with superscript~$\nu$.   Damping is proportional to the velocity~$v^\nu$ with constant~$k$.  The Laplacian governs near neighbour, particle-particle, quasi-elastic interaction forces, and a nonlinear reaction is characterised by~$f(u)$.  The externally imposed stochastic force field is~$\sigma \dot{W}(t,x)$ where $W(t,x)$~is an $L^{2}(D)$~valued Wiener process defined on some complete probability space~$(\Omega, \mathcal{F}, \mathbb{P})$, and is assumed to be of strength that scales according to $\sigma=\nu^\alpha$.   For exponent $\alpha=0$\,, the approximation of displacements~$u^{\nu}$, as $\nu\rightarrow 0$\,, is called the  infinite dimensional Smolukowski--Kramers approximation which has been proved valid in the limit by estimating the remainder term~\cite{CF05, CF06}. Our recent work~\cite{LvRob11} applied an averaging method to approximate the displacement field~$u^{\nu}$ for the case $\sigma=\nu^{\alpha}$ with exponent $0\leq \alpha\leq 1/2$\,.   Both of these methods are significantly complicated due to the coupling of displacement~$u$ and velocity~$v$ in the remainder term.

Here we apply a relatively simple method to derive suitable approximations for equations~(\ref{e:SWE1})--(\ref{e:SWE3}) with $\sigma=\nu^{\alpha}$, for exponent $\alpha\in [0,1)\cup (1, \infty)$.  We apply the following  useful  splitting  of the velocity~$u^{\nu}_{t}$,
\begin{equation}
u_{t}^{\nu}(t)=\frac{1}{\nu}\bar{v}_{1}^{\nu}(t)+v_{2}^{\nu}(t)+\nu^{\alpha-1/2}\bar{v}_{3}^{\nu}(t),
\end{equation}
to avoid directly estimating  the remainder  term.  The three parts of the above splitting are the initial value part, the mean value part, and the diffusion part, respectively. This decomposition gives a clear structure for the displacement~$u^{\nu}$ with
\begin{equation*}
u^{\nu}(t)-u_{0}=\frac{1}{\nu}\int_{0}^{t}\bar{v}^{\nu}_{1}(s)\,ds+\int_{0}^{t}\bar{v}_{2}^{\nu}(s)\,ds+\nu^{\alpha-1/2}\int_{0}^{t}\bar{v}_{3}^{\nu}(s)\,ds\,.
\end{equation*}
 The parts~$\bar{v}_{1}^{\nu}$ and~$\bar{v}_{3}^{\nu}$ satisfy linear equations, and section~\ref{sec:App} establishes
\begin{equation*}
\frac{1}{\nu}\int_{0}^{t}\bar{v}^{\nu}_{1}(s)\,ds=\mathcal{O}(\nu)\quad\text{and}\quad
\nu^{\alpha-1/2}\int_{0}^{t}\bar{v}_{3}^{\nu}\,ds=\mathcal{O}(\nu^{\alpha})
\quad\text{as }\nu\rightarrow 0\,.
\end{equation*}
The mean part $\bar{v}_{2}^{\nu}(t)$ of the velocity is~$\mathcal{O}(1)$ as $\nu\rightarrow 0$ for $t\in [0, T]$ with any fixed time~$T$.
Then for small~$\nu$, section~\ref{sec:App} determines which term is a high order term and gives an asymptotic approximation of the displacement~$u^{\nu}$.

Here one interesting case is when the exponent $\alpha=1$\,. In this case there are two terms with the same order~$\mathcal{O}(\nu)$  as $\nu\rightarrow 0$\,. Then if we keep all the $\mathcal{O}(\nu)$~terms, the approximation to the displacement~$u^{\nu}$ is just itself which is no modelling simplification. This case will be discussed further research.

Because of its motivation by physical continuum problems of wave motion in some random media~\cite{Chow81},
the system (\ref{e:SWE1})--(\ref{e:SWE3}) is called a stochastic wave equation. For small~$\nu$ and the particular case of $\sigma=\nu^{1/2}$\,,  Lv and Wang~\cite{LW08, WL10} studied the limit behaviour as $\nu\rightarrow 0$: in this case the random dynamics  of~(\ref{e:SWE1})--(\ref{e:SWE3}) was proved to be described by that of the nonlinear heat equation
\begin{equation}
u_{t}(t,x)=\D u( t,x)+f(u(t,x)),\quad u(0,x)=u_{0}\,.
\end{equation}
This paper extends this earlier research by approximating the behaviour of the solution on finite time interval~$[0, T]$, $T>0$\,,   for the more general case of $\sigma=\nu^{\alpha}$ with any $\alpha\in [0,1)\cup (1, \infty)$. The tightness in the space~$C(0, T; L^{2}(D))$, compact in sense of probability, has been proved in previous work~\cite{CF06, WL10}. Consequently, here we just need to approximate the displacement~$u^{\nu}$ in a weak sense; that is, we  consider the approximation of the inner product~$\langle u^{\nu}, \phi\rangle$ in the space~$C(0, T)$ for testing function $\phi\in C^{2}(D\times[0, T])$ with $\phi$~vanishing on the boundary~$\partial D$.

Section~\ref{sec:Pre} first gives some preliminaries and the main result, Theorem~\ref{thm:main}. Then section~\ref{sec:App} details the proof.

%%%%%%%%%%%%%%
%%%%%%%%%%%%%%
%%%%%%%%%%%%%%

\renewcommand{\Gamma}{\partial D}

\section{Preliminary}\label{sec:Pre}
Let $D\subset \R^n$, $1\leq n\leq 3$\,, be a regular domain with
boundary~$\Gamma$. Denote by~$L^2(D)$ the Lebesgue space of square
integrable real valued functions on~$D$, which is a Hilbert space
with inner product
\begin{equation*}
\langle u, v\rangle=\int_Du(x)v(x)\,dx\,, \quad u, v\in L^2(D)\,.
\end{equation*}
Write the norm on $L^2(D)$ by $\|u\|_0=\langle u,
u\rangle^{1/2}$. Define the following abstract operator
\begin{equation*}
Au=-\D u\,,\quad u\in \operatorname{Dom}(A)=\{u\in L^2(D): \D u\in L^2(D)\,,
\ u|_{\Gamma}=0\}.
\end{equation*}
Denoted by $\{\lambda_k\}$, assume the eigenvalues of operator~$A$ satisfy
$0<\lambda_1\leq \lambda_2\leq \lambda_3\leq \cdots$\,, and
$\lambda_k\rightarrow\infty$ as $k\rightarrow\infty$\,. For any
$s\geq 0$\,, denote by $H^{s}(D)$ the usual Sobolev space~$W^{s, 2}(D)$ and   by~$H^{s}_0(D)$ the closure of~$C_{0}^{\infty}(D)$ in~$H^{s}(D)$.  In the space~$H_{0}^{s}(D)$ we use the equivalent norm
\begin{equation*}
\|u\|_s=\|A^{s/2}u\|_0\,, \quad u\in H^s_0(D).
\end{equation*}
We also denote the dual space of~$H_{0}^{s}$ by $H^{-s}$.
Here specify that the noise magnitude scales as $\sigma=\nu^{\alpha}$, $0<\nu\leq 1$\,, for exponent $\alpha\geq 0$\,, in equation~(\ref{e:SWE2}); that is, we consider the following stochastic equations
\begin{eqnarray}\label{e:abs-SWE1}
u^\nu_t&=&v^\nu, \quad u^\nu(0)=u_0\,,
\\ v^\nu_t&=&\frac{1}{\nu}[-v^\nu-Au^\nu+f(u^\nu)]
+\nu^{\alpha-1} \dot{W}\,, \quad  u^\nu_t(0)=u_1\,.\label{e:abs-SWE2}
\end{eqnarray}
Hereafter we non-dimensionalise the time scale with the drag rate so that, in effect, the drag coefficient is one.  We assume
 $\{W(t,x)\}_{t\in \R}$ is an $L^2(D)$-valued,
two sided, Wiener process, defined on a complete probability space
$(\Omega, \mathcal{F}, \{\mathcal{F}_t\}_{t\geq 0},
\mathbb{P})$ with covariance operator~$Q$ such that
\begin{equation*}
Qe_k=b_k e_k\,, \quad  k=1,2,\ldots\,,
\end{equation*}
where $\{e_k\}$ is a complete orthonormal system in~$H$ and $\{b_k\}$~is a bounded sequence of non-negative real numbers. Then the noise process~$W(t,x)$ has the spectral expansion
\begin{equation*}
W(t,x)=\sum_{k=1}^\infty \sqrt{b_k} e_k w_k(t),
\end{equation*}
where $w_k$ are real, mutually independent, standard scalar Brownian
motions~\cite{PZ92}. Further, we assume boundedness of the sums
\begin{equation}\label{Q}
 \operatorname{tr}Q=\sum^\infty_{k=1} b_k<\infty \quad\text{and}\quad
\sum^\infty_{k=1}\lambda_k b_k<\infty\,.
\end{equation}
\begin{assumption}\label{H}
For the nonlinearity~$f$ we assume
\begin{enumerate}
    \item $|f(s)|\leq C_1(1+|s|^3)$,  $|f'(s)|\leq C_2(1+|s|^2)$;
    \item   $F(s)\leq -C_3(|s|^4-1)$,
     $sf(s)\leq -C_4(F(s)-1)$;
\end{enumerate}
 for some positive constants~$C_i$, $i=1,2,3,4$,  and where  $F(s)=\int_0^sf(r)\,dr$\,.
\end{assumption}
One simple example satisfying these assumptions is $f(u)=u-u^3$.

Then we have the following theorem.
\begin{theorem}\label{thm:existence}
Assume that the boundedness~(\ref{Q}) and Assumption~\ref{H} hold. For any $(u_0, u_1)\in
H_0^1(D)\times L^2(D)$\,, there is a unique solution $(u^\nu, v^{\nu})$ to~(\ref{e:abs-SWE1})--(\ref{e:abs-SWE2}), with
\begin{equation}\label{sol}
u^\nu\in L^2(\Omega, C(0, T; H^1_0(D))) \quad\text{and}\quad
  v^\nu\in L^2(\Omega, C(0, T;L^2(D))),
\end{equation}
for any $T>0$\,. Moreover, for any $T>0$ there is a positive constant~$C_T$ which is independent of~$\nu$ such that the expectation
\begin{equation*}
\mathbb{E}\sup_{0\leq t\leq T}\|u^\nu(t)\|_1 \leq C_T(\|u_0\|_{1}^2+\|u_1\|_{0}^2)\,,
\end{equation*}
and $\{u^\nu\}_{0<\nu\leq 1}$ is tight in the space~$C(0, T; L^2(D))$\,.
\end{theorem}
\begin{proof}
To prove the existence of the solution we define
\begin{equation*}
\mathcal{A}=\begin{bmatrix}
    0, &  \operatorname{id}_{L^2({D})}  \\
    \frac{1}{\nu}\D, &  -\frac{1}{\nu}\\
\end{bmatrix} ,\quad  F(u^\nu, v^\nu)=
\begin{bmatrix}
    0   \\
    \frac{1}{\nu}f(u^\nu)  \\
\end{bmatrix}
\quad\text{and}\quad
\mathcal{W}(t)=
\begin{bmatrix}
    0    \\
\frac{\nu^{\alpha-1/2}}{\sqrt{\nu}}W(t)
\end{bmatrix} .
\end{equation*}
Let $\Phi=(u^\nu, v^\nu)$, so equation (\ref{e:abs-SWE1})--(\ref{e:abs-SWE2}) can be rewritten in the following abstract stochastic evolutionary form
\begin{equation}\label{e:ABFSWE}
\dot{\Phi}=\mathcal{A}\Phi+F(\Phi)
+\dot{\mathcal{W}}\,, \quad \Phi(0)=(u_0, u_1).
\end{equation}
Notice that operator~$\mathcal{A}$ generates a strong continuous semigroup and the nonlinearity~$F$ is locally Lipschitz continuous, then by a standard method for stochastic evolutionary equations~\cite{PZ92} we have the first part of the theorem.

For $0\leq\alpha<1/2$ the energy  estimate for~($u^\nu$, $v^\nu$) and tightness result can be  obtained via a similar argument to that of Cerrai and Freidlin~\cite{CF06}, and for $\alpha\geq 1/2$ the energy estimate and tightness were obtained by Lv and Wang~\cite{LW08}. The proof is complete.
\end{proof}

In the following approach we need the following lemma on weak convergence of a sequence of functions due to Lions~\cite{Lions}.
\begin{lemma}\label{lem:convergence}
For any given functions~$h_{\nu}$ and $h\in L^{p}([0, T]\times D)$ $(1<p<\infty)$, if
$\|h_{\nu}\|_{L^{p}([0, T]\times D)}\leq C$
for some positive constant~$C$, and $h_{\nu}\rightarrow h$  on $[0, T]\times D$ almost everywhere as $\nu\rightarrow 0$\,, then $h_{\nu}\rightarrow h$ weakly in $L^{p}([0, T]\times D)$.
\end{lemma}

Now we give the  main theorem on the approximation of the displacement~$u^{\nu}$ in our stochastic wave equation.
\begin{theorem}\label{thm:main}
Assume that the boundedness~(\ref{Q}) and Assumption~\ref{H} hold, and $(u_{0}, u_{1})\in H_{0}^{1}(D)\times L^{2}(D)$\,.  If exponent $0\leq \alpha<1$\,,  for any $T>0$\,,  and for small $\nu>0$\,,  then with probability one
\begin{equation*}
\nu^{-\alpha}\|u^{\nu}-\bar{u}^{\nu}\|_{C(0, T; L^{2}(D))}\rightarrow 0\,,\quad \text{as } \nu\rightarrow 0 \,,
\end{equation*}
  with  the approximation~$\bar{u}^{\nu}$  solving the stochastic  nonlinear heat equation
\begin{equation}
\bar{u}^{\nu}_{t}=\D \bar{u}^{\nu}+f(\bar{u}^{\nu})+\nu^{\alpha}\dot{W}\,,\quad \bar{u}^{\nu}(0)=u_{0}\,.
\end{equation}
Conversely, if $\alpha>1$\,,
\begin{equation*}
\nu^{-1}\|u^{\nu}-\bar{u}^{\nu}\|_{C(0, T; L^{2}(D))}\rightarrow 0\,,\quad \text{as } \nu\rightarrow 0\,,
\end{equation*}
with the approximation~$\bar{u}^{\nu}$ solving  the deterministic nonlinear wave equation
\begin{equation}
\nu\bar{u}_{tt}+\bar{u}^{\nu}_{t}=\D \bar{u}^{\nu}+f(\bar{u}^{\nu})\,, \quad \bar{u}^{\nu}(0)=u_{0}\,, \ \bar{u}_{t}^{\nu}(0)=u_{1}\,.
\end{equation}

\end{theorem}

%%%%%%%%%%%%%%%%%%%%%%
%%%%%%%%%%%%%%%%%%%%%%
%%%%%%%%%%%%%%%%%%%%%%

\section{Approximation: proof of Theorem~\ref{thm:main}}\label{sec:App}
By Theorem~\ref{thm:existence},  $\{u^{\nu}(t)\}_{0<\nu\leq 1}$~is tight in the space~$C(0, T; L^{2}(D))$, so we  approximate the displacement~$u^{\nu}$ in a weak sense:  we approximate~$\langle u^{\nu}, \phi\rangle $ for any $\phi\in C^{2}([0, T]\times D)$ with~$\phi|_{\Gamma}=0$\,.

In order to avoid the coupling between the displacement~$u^{\nu}$ and the velocity~$v^{\nu}$,  we scale the velocity field as
\begin{equation}\label{e:transf}
\tilde{v}^{\nu}=\nu v^{\nu}.
\end{equation}
Then
\begin{eqnarray*}
u^{\nu}_{t}&=&\frac{1}{\nu}\tilde{v}^{\nu},\quad u^{\nu}(0)=u_{0}\,,\\
\tilde{v}^{\nu}_{t}&=&-\frac{1}{\nu}\tilde{v}+\D u^{\nu}+f(u^{\nu})+\nu^{\alpha}\dot{W}\,,\quad \tilde{v}^{\nu}(0)=\nu u_{1}\,.
\end{eqnarray*}
Further, we make the decomposition
\begin{eqnarray}\label{e:split}
\tilde{v}^{\nu}&=&\bar{v}^{\nu}_{1}+\nu\bar{v}_{2}^{\nu}+\nu^{\alpha+1/2}\bar{v}_{3}^{\nu} \,,
\\
\text{where}\quad
\bar{v}_{1,t}^{\nu}&=&-\frac{1}{\nu}\bar{v}_{1}^{\nu}\,,\quad \bar{v}_{1}^{\nu}(0)=\nu u_{1}\,,\label{e:v1}\\
\bar{v}_{2,t}^{\nu}&=&-\frac{1}{\nu}[\bar{v}^{\nu}_{2}-\D u^{\nu}-f(u^{\nu})]\,,\quad \bar{v}_{2}^{\nu}(0)=0\,,\label{e:v2}\\
\bar{v}_{3,t}^{\nu}&=&-\frac{1}{\nu}\bar{v}^{\nu}_{3}+\frac{1}{\sqrt{\nu}}\dot{W}\,,\quad \bar{v}_{3}^{\nu}(0)=0\,.\label{e:v3}
\end{eqnarray}
Then
\begin{equation}\label{e:tilde-u}
u_{t}^{\nu}=\frac{1}{\nu}\bar{v}_{1}^{\nu}+\bar{v}_{2}^{\nu}+\nu^{\alpha-1/2}\bar{v}_{3}^{\nu}\,,\quad \tilde{u}^{\nu}(0)=u_{0}\,.
\end{equation}
The decomposition of $\tilde{v}^{\nu}$ makes the problem easier. The two \sde{}s \eqref{e:v1}~and~\eqref{e:v3} for the two components~$\bar{v}_{1}^{\nu}$ and~$\bar{v}^{\nu}_{3}$ are just linear \sde{}s whose properties are well known. The properties of~$\bar{v}_{2}^{\nu}$  can be derived straightforwardly from the \pde~\eqref{e:v2} by the estimates in Theorem~\ref{thm:existence}.  We state the following results.
\begin{lemma}\label{lem:v1v3}
Assume that the boundedness~(\ref{Q}) and Assumption~\ref{H} hold. Let  $u_{1}\in L^{2}(D)$,  then for any $\phi\in C^{2}([0, T]\times D)$ with $\phi|_{\Gamma}=0$\,,
\begin{eqnarray} && \label{e:v1conv}
\frac{1}{\nu}\int_{0}^{t}\langle \bar{v}_{1}^{\nu}(s), \phi(s)\rangle\,ds\rightarrow 0\,,\quad 0\leq t\leq T\,,
\\ \text{and}&&
\label{e:v3conv}
\nu^{-1/2}\int_{0}^{t}\langle \bar{v}_{3}^{\nu}(s),\phi(s)\rangle\,ds\rightarrow \int_{0}^{t}\langle \phi(s), dW(s)\rangle, \quad 0\leq t\leq T\,,
\end{eqnarray}
in $L^{2}(\Omega)$, as $\nu\rightarrow 0$\,.
\end{lemma}
 \begin{proof}
The proof is direct.  First,
\begin{equation*}
\bar{v}_{1}^{\nu}(t)=\nu u_{1} e^{-{t}/{\nu}}.
\end{equation*}
Then for $\phi\in C^{2}([0, T]\times D)$
\begin{eqnarray*}
\frac{1}{\nu}\int_{0}^{t}\langle\bar{v}_{1}^{\nu}(s), \phi(s)\rangle\,ds&=&\int_{0}^{t}\langle u_{1}, \phi(s)\rangle e^{-s/\nu}\,ds\\
&=&\nu\int_{0}^{t/\nu}\langle  u_{1}, \phi(\nu \tau)\rangle e^{-\tau}\,d\tau\rightarrow 0 \,,\quad\text{as } \nu\rightarrow 0\,,
\end{eqnarray*}
 uniformly on $[0, T]$\,, which yields the first convergence. Second, for any $\phi\in C^{2}([0, T]\times D)$ with $\phi|_{\Gamma}=0$ by equation (\ref{e:v3}),
\begin{eqnarray}
\nu^{-1/2}\int_{0}^{t}\langle\bar{v}^{\nu}_{3}(s),\phi(s)\rangle\,ds&=&-\sqrt{\nu}\int_{0}^{t}\langle\bar{v}_{3,t}^{\nu}(s),\phi(s)\rangle\,ds+\int_{0}^{t}\langle \phi(s),  dW(s)\rangle\nonumber\\
&=&-\sqrt{\nu}\langle \bar{v}_{3}^{\nu}(t), \phi(t)\rangle+\sqrt{\nu}\int_{0}^{t}\langle \bar{v}_{3}^{\nu}(s), \phi_{t}(s)\rangle\,ds\nonumber\\
&&{}+\int_{0}^{t}\langle \phi(s), dW(s)\rangle\,.\label{e:v3appr}
\end{eqnarray}
 So it remains to show that $\bar{v}_{3}^{\nu}(t)$ is uniformly  bounded in the space~$L^{2}(\Omega, L^{2}(D))$.  By equation~(\ref{e:v3}),  applying It\^o formula to~$\|\bar{v}_{3}^{\nu}\|_{0}^{2}$ gives
 \begin{equation*}
\frac12\frac{d}{dt}\|\bar{v}_{3}^{\nu}(t)\|_{0}^{2}=-\frac{1}{\nu}\|\bar{v}_{3}^{\nu}\|^{2}_{0}+\frac{1}{2\nu}\operatorname{tr}Q+\frac{1}{\sqrt{\nu}}\langle \bar{v}_{3}^{\nu}, \dot{W} \rangle.
 \end{equation*}
 Then by the Gronwall lemma
 \begin{equation*}
\mathbb{E}\|\bar{v}_{3}^{\nu}(t)\|^{2}_{0}\leq \operatorname{tr}Q\,, \quad t\geq 0\,.
 \end{equation*}
 The proof is complete.
 \end{proof}

\begin{lemma}\label{lem:v2}
Assume the conditions in Theorem~\ref{thm:existence} holds, then there is a parameter~$\nu$ and independent positive constant~$C_{T}$ such that
\begin{equation*}
\mathbb{E}\|\bar{v}^{\nu}_{2}(t)\|_{-1}\leq C_{T}\,,\quad 0\leq t\leq T\,.
\end{equation*}
\end{lemma}
\begin{proof}
For any $\psi\in H_{0}^{1}(D)$\,, from equation~(\ref{e:v2})
\begin{equation*}
\frac{d}{dt}\langle \bar{v}^{\nu}_{2}, \psi \rangle=-\frac{1}{\nu}\langle \bar{v}^{\nu}_{2}, \psi  \rangle-\frac{1}{\nu}\langle \nabla u^{\nu}, \nabla\psi\rangle+\frac{1}{\nu}\langle f(u^{\nu}), \psi\rangle.
\end{equation*}
Then
\begin{equation*}
\langle \bar{v}_{2}^{\nu}(t), \psi\rangle=
\frac{1}{\nu}e^{-{t}/{\nu}}\int_{0}^{t}e^{{s}/{\nu}}\left[-\langle \nabla u^{\nu}(s), \nabla\psi\rangle+\langle f(u^{\nu}(s)), \psi \rangle \right]\,ds\,.
\end{equation*}
By the estimates in Theorem~\ref{thm:existence} and the embedding $H_{0}^{1}(D)\subset L^{6}(D)$ for $1\leq n\leq 3$\,,
\begin{equation*}
\mathbb{E}|\langle \bar{v}^{\nu}_{2}(t), \psi\rangle|\leq C_{T}\|\psi\|_{1}\,,\quad 0\leq t\leq T\,.
\end{equation*}
The proof is complete.

\end{proof}

From the above lemma we prove the main Theorem~\ref{thm:main}.  First, for any $\kappa>0$\,, by the tightness of displacement~$u^{\nu}$ in the space~$C(0, T; L^{2}(D))$, there is a compact set $B_{\kappa}\subset C(0, T; L^{2}(D))$ such that
\begin{equation}\label{e:event1}
\mathbb{P}\{ u^{\nu}\in B_{\kappa}\}\geq 1-\kappa/2\,.
\end{equation}
By the Markov inequality and the estimate in Lemma~\ref{lem:v2}, for any $\kappa>0$ there is a positive constant $C_{T}^{\kappa}$ such that
\begin{equation}\label{e:event2}
\mathbb{P}\left\{\|\bar{v}^{\nu}_{2}(t)\|_{-1}\leq C_{T}^{\kappa}\|\right\}\geq 1-\kappa/2\,.
\end{equation}
Then for any $\kappa>0$\,, define a probability space~$(\Omega_{\kappa}, \mathcal{F}_{\kappa}, \mathbb{P}_{\kappa})$
\begin{equation*}
\Omega_{\kappa}=\{\omega\in\Omega: \text{events~(\ref{e:event1}) and~(\ref{e:event2}) hold}\}, \quad  \mathcal{F}_{\kappa}=\{F\cap \Omega_{\kappa}: F\in\mathcal{F} \},
\end{equation*}
and for any $F\in\mathcal{F}_{\kappa}$
\begin{equation*}
\mathbb{P}_{\kappa}(F)=\frac{\mathbb{P}(F\cap\Omega_{\kappa})}{\mathbb{P}(\Omega_{\kappa})}\,.
\end{equation*}
In the following we restrict our problem to the above new probability space.  For any $\omega\in\Omega_{\kappa}$\,, the convergence~(\ref{e:v1conv}) still holds. The convergence~(\ref{e:v3conv}) is in the $L^{2}(\Omega)$~sense which yields the convergence for $\mathbb{P}$~almost all $\omega\in\Omega$\,, then we also have the  convergence~(\ref{e:v3conv}) for~$\mathbb{P}_{\kappa}$ almost all $\omega\in\Omega_{\kappa}$\,. So we can assume that for all $\omega\in\Omega_{\kappa}$\,, the convergence~(\ref{e:v3conv}) holds.

Furthermore, we establish the limit
\begin{equation}\label{e:f-conv}
f(u^{n})\rightarrow f(u)\ \text{ weakly in }  L^{2}(0, T; L^{2}(D))
\end{equation}
for any $u^{n}\rightarrow u$ in $C(0, T; L^{2}(D))$\,.  By the embedding $H^{1}_{0}(D)\subset L^{6}(D)$  we have $\|f(u^{n})\|_{L^{2}(0, T; L^{2}D)}\leq C_{T}$ for some positive constant $C_{T}$\,, and by the strong convergence of $u^{n}\rightarrow u$\,, $f(u^{n}(t,x))\rightarrow f(u(t,x))$ on $[0, T]\times D$ almost everywhere.  Then Lemma~\ref{lem:convergence} gives the limit.

%By the compactness of $B_{\kappa}$\,, for any $\omega\in\Omega_{\kappa}$ and any sequence, there is a subsquence  $\nu\rightarrow0$ as $n\rightarrow\infty$ and $u\in C(0, T; L^{2}(D))$ such that
%\begin{equation}\label{e:strongconvergence}
%\tilde{u}^{\nu}\rightarrow u\,, \quad \text{in}\ C(0, T; L^{2}(D))\quad \text{as}\  n\rightarrow\infty\,.
%\end{equation}
Next we give an asymptotic approximation  to the displacement~$u^{\nu}$. For this we consider $\langle u^{\nu}(t), \phi(t)\rangle$ with $\phi\in C^{2}([0, T]\times D)$ and $\phi|_{\Gamma}=0$\,.  From equation~(\ref{e:tilde-u})
\begin{eqnarray*}
&&\langle u^{\nu}(t), \phi(t)\rangle-\langle u_{0}, \phi(0)\rangle-\int_{0}^{t}\langle u^{\nu}(s), \phi_{t}(s)\rangle\,ds\\
 &=&\frac{1}{\nu}\int_{0}^{t}\langle \bar{v}_{1}^{\nu}(s), \phi(s)\rangle\,ds+\int_{0}^{t}\langle\bar{v}_{2}^{\nu}(s), \phi(s)\rangle\,ds +\nu^{\alpha-1/2}\int_{0}^{t}\langle \bar{v}_{3}^{\nu}(s), \phi(s)\rangle\,ds\,.
 \end{eqnarray*}
From  equation~(\ref{e:v2})
\begin{eqnarray*}
\int_{0}^{t}\langle\bar{v}^{\nu}_{2}(s), \phi(s) \rangle\,ds&=&\int_{0}^{t}\langle u^{\nu}(s), \Delta \phi(s)\rangle\,ds+\int_{0}^{t}\langle f(u^{\nu}(s)), \phi(s)\rangle\,ds\\&&{}-\nu\langle \bar{v}_{2}^{\nu}(t), \phi(t)\rangle+\nu\int_{0}^{t}\langle \bar{v}_{2}^{\nu}(s), \phi_{t}(s)\rangle\,ds\,.
\end{eqnarray*}
By the definition of $\Omega_{\kappa}$, and Lemmas~\ref{lem:v1v3} and~\ref{lem:v2},
\begin{equation*}
\frac{1}{\nu}\int_{0}^{t}\langle\bar{v}^{\nu}_{1}(s), \phi(s) \rangle\,ds=\mathcal{O}(\nu)\,,\quad \nu\langle\bar{v}^{\nu}_{2}(t), \phi(t) \rangle=\mathcal{O}(\nu),
\end{equation*}
and
\begin{equation*}
\nu\int_{0}^{t}\langle \bar{v}_{2}^{\nu}(s), \phi_{t}(s)\rangle\,ds=\mathcal{O}(\nu)\,.
\end{equation*}
Further, by~(\ref{e:v3appr})
\begin{equation*}
\nu^{\alpha-1/2}\int_{0}^{t}\langle\bar{v}^{\nu}_{3}(s), \phi(s) \rangle\,ds=\nu^{\alpha}\int_{0}^{t}\langle \phi(s), dW(s) \rangle+\mathcal{O}(\nu^{\alpha+1/2})\,.
\end{equation*}
Then
\begin{eqnarray}
&&\langle u^{\nu}(t), \phi(t)\rangle-\langle u_{0}, \phi(0)\rangle-\int_{0}^{t}\langle u^{\nu}(s), \phi_{t}(s)\rangle\,ds-\int_{0}^{t}\langle u^{\nu}(s),\Delta\phi(s) \rangle\,ds\nonumber\\&&{}-\int_{0}^{t}\langle f(u^{\nu}(s)), \phi(s)\rangle\,ds\nonumber\\
 &=&\nu^{\alpha-1/2}\int_{0}^{t}\langle \bar{v}_{3}^{\nu}(s), \phi(s)\rangle\,ds+\frac{1}{\nu}\int_{0}^{t}\langle \bar{v}_{1}^{\nu}(s), \phi(s)\rangle\,ds \nonumber\\&&{}-\nu\langle \bar{v}_{2}^{\nu}(t), \phi(t)\rangle+\nu\int_{0}^{t}\langle \bar{v}_{2}^{\nu}(s), \phi_{t}(s)\rangle\,ds \nonumber\\
 &=&\nu^{\alpha}\int_{0}^{t}\langle \phi(s), dW(s)\rangle+\mathcal{O}(\nu^{\alpha+1/2})+\mathcal{O}(\nu)\,. \label{e:expansion}
 \end{eqnarray}
 Now for exponent $0\leq \alpha<1$\,,
noticing the convergence (\ref{e:f-conv}) and neglecting the~$o(\nu^{\alpha})$ terms we have the following equation:
\begin{eqnarray*}
&&\langle \bar{u}^{\nu}(t), \phi(t)\rangle-\langle u_{0}, \phi(0)\rangle-\int_{0}^{t}\langle \bar{u}^{\nu}(s), \phi_{t}(s)\rangle\,ds-\int_{0}^{t}\langle \bar{u}^{\nu}(s),\Delta\phi(s) \rangle\,ds\\&&{}-\int_{0}^{t}\langle f(\bar{u}^{\nu}(s)), \phi(s)\rangle\,ds=\nu^{\alpha}\int_{0}^{t}\langle \phi(s), dW(s)\rangle\,.
\end{eqnarray*}
Then we deduce the following approximation equation holds:
\begin{equation}\label{e:SHE}
\bar{u}^{\nu}_{t}=\D\bar{u}^{\nu}+f(\bar{u}^{\nu})+\nu^{\alpha}\dot{W}\,,\quad \bar{u}^{\nu}(0)=u_{0}\,.
\end{equation}

For exponent $\alpha>1$  we need to show the rate of decay in~$\nu$ of  $v_{1}^{\nu}$~and~$v_{2}^{\nu}$ as $\nu\rightarrow 0$\,. First, as $\{u^{\nu}\}_{\nu}\subset B_{\kappa}$ which is compact in $C(0, T; L^{2}(D))$\,, for any sequence, there are a subsequence, say $u^{\nu_{n}}$\,, with $\nu_{n}\rightarrow 0$ as $n\rightarrow\infty$\,,   and $\bar{u}\in C(0, T; L^{2}(D))$ with $\bar{u}(0)=u_{0}$ such that
\begin{equation*}
u^{\nu_{n}}\rightarrow \bar{u}\,, \quad n\rightarrow \infty\,.
\end{equation*}

First we assume
\begin{equation*}
\Delta\bar{u}+f(\bar{u})\neq 0 \quad \text{in } H^{-1}
\quad \text{for all }t\in [0, T].
\end{equation*}
Then by equation~(\ref{e:v2}), for any $\phi\in C^{2}([0, T]\times D)$ with $\phi|_{\Gamma}=0$
\begin{eqnarray*}
&&\nu_{n}\langle\bar{v}_{2}^{\nu_{n}}(t), \phi(t) \rangle-\nu_{n}\int_{0}^{t}\langle \bar{v}_{2}^{\nu_{n}}(s), \phi_{t}(s)\rangle\,ds\\
&=& -\int_{0}^{t}\langle\bar{v}_{2}^{\nu_{n}}(s), \phi(s) \rangle\,ds+\int_{0}^{t}\langle u^{\nu_{n}}(s), \D\phi(s)\rangle\,ds+
\int_{0}^{t}\langle f(u^{\nu_{n}}(s)), \phi(s) \rangle\,ds\,.
\end{eqnarray*}
By the estimates in Lemma~\ref{lem:v2}, and by the convergence~(\ref{e:f-conv}),
\begin{equation*}
\int_{0}^{t}\langle
\bar{v}_{2}^{\nu_{n}}(s), \phi(s)\rangle\,ds\rightarrow \int_{0}^{t}\langle \bar{u}(s), \D\phi(s) \rangle\,ds+\int_{0}^{t}\langle f(\bar{u}(s)), \phi(s)\rangle\,ds\neq 0
\end{equation*}
as $n\rightarrow\infty$\,. This shows that
$\bar{v}_{2}^{\nu}(t)$ does not converge to zero as $\nu\rightarrow 0$ for $0<t\leq T$\,. Then
\begin{equation}
\nu\langle \bar{v}_{2}^{\nu}(t), \phi(t)\rangle=\mathcal{O}(\nu)\quad \text{and} \quad
\nu\int_{0}^{t}\langle\bar{v}_{2}^{\nu}(s), \phi_{t}(s) \rangle\,ds=\mathcal{O}(\nu)
\end{equation}
instead of $o(\nu)$ as $\nu\rightarrow 0$\,.

Further  from equation~(\ref{e:v1})\,, for any $\phi\in C^{2}([0, T]\times D)$ with $\phi|_{\Gamma}=0$\,,
\begin{eqnarray*}
\frac{1}{\nu}\langle\bar{v}_{1}^{\nu}(t), \phi(t) \rangle-\langle u_{1}, \phi(0)\rangle-\frac{1}{\nu}\int_{0}^{t}\langle\bar{v}_{1}^{\nu}(s),\phi_{t}(s) \rangle\,ds=-\frac{1}{\nu^{2}}\int_{0}^{t}\langle \bar{v}_{1}^{\nu}(s), \phi(s)\rangle\,ds\,.
\end{eqnarray*}
Then
\begin{equation}
\frac{1}{\nu^{2}}\int_{0}^{t}\langle\bar{v}_{1}^{\nu}(s), \phi(s) \rangle\,ds\rightarrow \langle u_{1}, \phi(0) \rangle\,,\quad \nu\rightarrow 0\,.
\end{equation}
Then for exponent $\alpha>1$\,, in  the asymptotic expansion~(\ref{e:expansion}) for small~$\nu$, neglecting the $o(\nu)$~term consisting of~$\bar{v}_{3}^{\nu}$, and by the  transformation~(\ref{e:transf}) and decomposition~(\ref{e:split}),
\begin{eqnarray*}
&&\langle u^{\nu}(t), \phi(t)\rangle-\langle u_{0}, \phi(0)\rangle-\int_{0}^{t}\langle u^{\nu}(s), \phi_{t}(s)\rangle\,ds-\int_{0}^{t}\langle u^{\nu}(s),\Delta\phi(s) \rangle\,ds\\&&{}-\int_{0}^{t}\langle f(u^{\nu}(s)), \phi(s)\rangle\,ds\nonumber\\
 &=&\frac{1}{\nu}\int_{0}^{t}\langle \bar{v}_{1}^{\nu}(s), \phi(s)\rangle\,ds -\nu\langle \bar{v}_{2}^{\nu}(t), \phi(t)\rangle+\nu\int_{0}^{t}\langle \bar{v}_{2}^{\nu}(s), \phi_{t}(s)\rangle\,ds \\
 &=&-\nu\langle v^{\nu}(t), \phi(t)\rangle+\nu\int_{0}^{t}\langle v^{\nu}(s), \phi_{t}(s) \rangle\,ds
 +\nu\langle u_{1}, \phi(0)\rangle
 \\&&{}
 +\nu^{\alpha+1/2}\langle \bar{v}^{\nu}_{3}(t), \phi(t)\rangle
 -\nu^{\alpha+1/2}\int_{0}^{t}\langle \bar{v}^{\nu}_{3}(s), \phi_{t}(s)\rangle\,ds\,.
 \end{eqnarray*}
Then noticing that velocity~$v^{\nu}=u^{\nu}_{t}$\,, and neglecting the $\mathcal{O}(\nu^{\alpha+1/2})$~terms,  we have the following approximation equation
\begin{equation}\label{e:WEn}
\nu \bar{u}^{\nu}_{tt}+\bar{u}^{\nu}_{t}=\D \bar{u}^{\nu}+f(\bar{u}^{\nu})\,.
\end{equation}
Second if for $t\in [0, T]$
\begin{equation*}
\Delta\bar{u}+f(\bar{u})=0\,, \quad \bar{u}(0)=u_{0}
\end{equation*}
in $H^{-1}$.  Then $\bar{u}$~is a stationary solution of~(\ref{e:WEn})\,.

The above approximation is in the sense of $\mathbb{P}_{\kappa}$~almost surely. Then by the arbitrary choice of~$\kappa$, and the well-posedness of  (\ref{e:SHE}) and (\ref{e:WEn}),
this establishes the approximation with~$\mathbb{P}$~probability one.

This completes our proof of the main Theorem~\ref{thm:main} on the approximations of the stochastic wave equation~\eqref{e:SWE1}--\eqref{e:SWE3} for different scaling of the noise process.

%%%%%%%%%%%%%%%%%%%%%%
%%%%%%%%%%%%%%%%%%%%%%
%%%%%%%%%%%%%%%%%%%%%%

\paragraph{Acknowledgements} This research was supported by the
NSFC grant No.~10901083 and by the Australian Research Council grants DP0774311 and DP0988738.

\end{document}